\documentclass[11pt,oneside]{article}
\usepackage[centertags]{amsmath}
\usepackage{amssymb}
\usepackage{amsthm}
\usepackage{hyperref}
\newtheorem{thm}{Theorem}[section]

\newtheorem{prop}[thm]{Proposition}
\newtheorem{cor}[thm]{Corollary}
\newtheorem{rem}[thm]{Remark}
\newcommand{\punt}{\boldsymbol{.}}
\newcommand{\rv}{\hbox{r.v.}}
\begin{document}

\title{A new approach to Sheppard's corrections}
\author{E. Di Nardo \thanks{Dipartimento di Matematica e Informatica,
Universit\`a degli Studi della Basilicata, Viale dell'Ateneo Lucano 10, 85100
Potenza,
Italia, elvira.dinardo@unibas.it}}
\date{\today}
\maketitle

\begin{abstract}
A very simple closed-form formula for Sheppard's corrections is
recovered by means of the classical umbral calculus. By means of
this symbolic method, a more general closed-form formula for
discrete parent distributions is provided and the generalization
to the multivariate case turns to be straightforward. All these
new formulae are particularly suited to be implemented in any
symbolic package.
\end{abstract}

\textsf{\textbf{keywords}:
raw moment, grouped moment, Sheppard's correction, umbral calculus, Bernoulli
polynomials}
\section{Introduction}
In the real world, continuous variables are observed and recorded
in finite precision through a rounding or coarsening operation,
i.e. a grouping rule. A compromise between the desire to know and
the cost of knowing is then a necessary consequence. The
literature on grouped data spans different research areas, see for
example \cite{Heitjan}. In particular, attention has been paid in
the literature to the computation of moments when data are grouped
into classes. Indeed, the method of moments performs estimation
less well than Maximum Likelihood, but it is useful in situation
when Maximum Likelihood is not feasible or has poor small sample
size performance.

The moments computed by means of the resulting grouped frequency
distribution are looked upon as a first approximation to the moments of
the parent distribution, but they suffer from the error committed
in grouping. The correction for grouping is a sum of two terms,
the first depending on the length of the grouping interval, the
second being a periodic function of the position. For a continuous
random variable ($\rv$), this very old problem was first discussed by
Thiele \cite{Thiele}, who studied the second term missing the
first, and then by Sheppard \cite{Sheppard} who studied the first
term, missing the second. Both Bruns \cite{Bruns} and Fisher
\cite{Fisher} proved that the second term can be neglected under suitable
hypothesis and so for using Sheppard's corrections,
that are nowadays still employed. If $\tilde{a}_n$ denotes the $n$-th moment of
the grouped distribution, then the $n$-th raw moment $a_n$ of the
continuous parent distribution can be constructed
via Sheppard's corrections as
\begin{equation}
a_n = \sum_{j=0}^{n} {n \choose j} \left( 2^ {1-j} -1 \right) B_j
\, h^j \, \tilde{a}_{n - j}, \label{(Sh1)}
\end{equation}
where $\{B_j\}$ are the {\it Bernoulli numbers} \footnote{Many
characterizations of the Bernoulli numbers can be found in the
literature and each could be used to define these numbers. Here we
refer to the sequence of numbers such that $B_0=1$ and
$\sum_{k=0}^n  { n+1 \choose k}  B_k = 0$ for $n=0,1,2,\ldots.$}
and $h$ is the length of the grouping intervals. The derivation of
Sheppard's corrections was a popular topic in the first half of
last century, see \cite{Hald} for an historical account. This
because Sheppard deduced equation (\ref{(Sh1)}) by using a
suitable summation formula whose remainder term goes to zero when
the density function has high order contact with the $x$-axis at
both ends. So there was a considerable controversy on the set of
sufficient conditions to be required in order to use formula
(\ref{(Sh1)}). All these sufficient conditions can be removed, if
the rounding lattice is assumed to be random and the average is
made. Less research on moment corrections has appeared since the
definitive work of Kendall \cite{Kendall}, although recently the
appropriateness of Sheppard's corrections was re-examined in
connection with some applications, see for example \cite{Dempster}
and \cite{Vardeman}.

Grouping includes also censoring or splitting data into categories during
collection or publication, and so it does not only involve continuous
variables. The derivation of corrections for raw moments of a discrete
parent distribution followed a different path from Sheppard's corrections.
They were first given in the Editorial
of Vol.$1$, no. $1$, of {\it Annals of Mathematical
Statistics} (page 111). The method used to develop the general
formula was extremely laborious. Some years later, Craig
\cite{Craig} considerably reduced and simplified the derivation of
these corrections by using the logarithm of the moment generating
function, that is working on cumulants instead of moments. Craig
proposed the same method to derive formula (\ref{(Sh1)}),
stating these corrections on the average and so avoiding to require
any conditions on the parent distribution. At the moment, his
method represents the most general way to find such corrections,
both for continuous and for discrete parent distributions.

In this paper, we propose to overcome Craig's methods by stating
Sheppard's corrections on the average through the employment of the
classical umbral calculus. This approach is partially motivated by
a similar employment of the classical umbral calculus in wavelets
theory \cite{Saliani, Shen}. Indeed, the reconstruction of the full
moment information $a_n$ in (\ref{(Sh1)}) through the grouped moments
$\tilde{a}_n$ can be seen as some kind of multilevel analysis, like
in wavelets analysis. But the paper is inspired also from the belief
that the classical umbral calculus can be fruitful
used in statistics, coming out the side authentically algebraic
of many techniques commonly used to manage number sequences
related to $\rv$'s.

The umbral calculus was studied by Rota and his collaborators for
long time \cite{DiBuc}. The version introduced in \cite{SIAM}
represents a new way of dealing with number sequences, which is
represented by a symbol $\alpha,$ called an {\it umbra}. More
precisely, an unital sequence $1, a_1, a_2, \ldots$ is associated
to the sequence $1, \alpha, \alpha^2, \ldots$ of powers of
$\alpha$ through a linear functional $E$ that looks like the
expectation of a $\rv$. So the classical umbral calculus is no
more than a symbolic tool by which handling number sequences. An
umbra looks as the framework of a $\rv$ with no reference to any
probability space, someway getting closer to statistical methods.
Compared with previous symbolic methods employed in statistics,
see for example \cite{Andrews} and \cite{McCullagh}, by a
theoretical point of view it has the advantage to reduce the
combinatorics of symmetric functions, commonly used by
statisticians, to few relations which cover a great variety of
calculations \cite{CompStat}. By a computational point of view,
the efficiency of umbral calculus in manipulating expressions
involving $\rv$'s has been tested on the theory of $k$-statistics
\cite{Bernoulli} and their generalizations \cite{Statcomp} as well
as in manipulating $U$-statistics and product moments of sample
moments \cite{CompStat}. Recently, also the free cumulant theory
has been approached by means of this new syntax \cite{Petrullo},
showing promises for future developments \cite{Oliva}.

Finally, the employment of the classical umbral calculus in
corrections of moments for grouped data has one more advantage.
Except for the papers of Craig \cite{Craig} and Baten
\cite{Baten}, no attention was paid to multivariate
generalizations of Sheppard's corrections, probably due to the
complexity of the resulting formulae. The notion of multiset is
the combinatorial device that, via the symbolic method, gives rise
to a closed-form formula for multivariate parent distributions
that could be implemented in few steps by using any symbolic
package.

The paper is structured as follows. Section 2 is provided for
readers unaware of the classical umbral calculus. Let us
underline that the theory of the classical umbral calculus
has now reached a more advanced level compared to the
elements here resumed. We have chosen to recall terminology,
notation and the basic definitions strictly necessary to deal with
the object of this paper. In particular, we recall the notion
of the Bernoulli umbra, as introduced in \cite{SIAM}. The Bernoulli
umbra is the keystone for the umbral handling of moment corrections of
grouped data. In Section 3, Sheppard's corrections are given and
extended to discrete parent distributions. For the continuous
case, the key is to represent integrals by means of suitable
umbral Bernoulli polynomials. For the discrete case, the key is
the generalization of the so-called {\it multiplication theorem} to umbral
Bernoulli polynomials. Differently from the
literature on this subject, where first the expressions of
corrected moments in terms of raw ones are deduced and then these
expressions are inverted by solving a linear system, here we
deduce directly the corrections on the moments, due to closed-form
formulae. Section 4 is devoted to the multivariate generalizations
of Sheppard's corrections. Some concluding remarks end the paper.
\section{Background on umbral calculus}
In the following, terminology, notation and some basic definitions
of the classical umbral calculus are recalled, as introduced by
Rota and Taylor in \cite{SIAM} and subsequently developed by Di
Nardo and Senato in \cite{Dinardo} and \cite{Dinardoeurop}. We
skip any proof: the reader interested in-depth analysis is
referred to the papers quoted belove.

The classical umbral calculus is a syntax consisting of the following
data:
\begin{description}
\item[{\it i)}] a set $A=\{\alpha,\beta, \ldots \},$ called the
{\it alphabet}, whose elements are named {\it umbrae}; \item[{\it
ii)}] the polynomial ring ${\cal R}={\mathbb R}[x]$ in the indeterminate $x$
with ${\mathbb R}$ the field of real numbers\footnote{For
the aim of this paper, we need something more of the usual commutative integral
domain whose quotient field is of characteristic zero,
required in \cite{SIAM}. Due to the framework we deal, it is more convenient
to refer directly to the field of real numbers ${\mathbb R}.$}; \item[{\it iii)}] a linear
functional $E: {\cal R}[A] \rightarrow {\cal R},$ called {\it evaluation},
such that $E[1]=1$ and $E[x^n \alpha^i \beta^j \cdots \gamma^k] = x^n E[\alpha^i]E[\beta^j]
\cdots E [\gamma^k]$ for any set of distinct umbrae in $A$ and for
$n,m,i,j,\ldots,k$ nonnegative integers (the so-called {\it
uncorrelation property}); \item[{\it iv)}] an element $\varepsilon
\in A,$ called {\it augmentation}, such that $E[\varepsilon^n] =
\delta_{0,n},$ for any nonnegative integer $n,$ where
$\delta_{i,j} = 1$ if $i=j,$ otherwise being zero; \item[{\it v)}]
an element $u \in A,$ called {\it unity} umbra, such that
$E[u^n]=1,$ for any nonnegative integer $n.$
\end{description}

A sequence $a_0=1,a_1,a_2, \ldots$ in ${\cal R}$ is umbrally represented
by an umbra $\alpha$ when
$$E[\alpha^i]=a_i, \quad \hbox{for} \,\, i=0,1,2,\ldots.$$
The elements $a_i$ are called {\it moments} of the umbra $\alpha,$
in analogy with the moments of a $\rv \, X.$ In particular, an umbra is
said to be \textit{scalar} if the moments are elements of ${\mathbb R}$
while it is said to be \textit{polynomial} if the moments are
polynomials of ${\mathbb R}[x].$

An umbral polynomial is a polynomial $p \in {\cal R}[A].$ The support of
$p$ is the set of all umbrae occurring in $p.$ If $p$ and $q$ are
two umbral polynomials then  $p$ and $q$ are {\it uncorrelated} if
and only if their supports are disjoint. We said that $p$ and $q$ are {\it
umbrally equivalent} if and only if
$$E[p] =E[q], \quad \hbox{\rm in symbols} \quad p \simeq q.$$

It is possible that two distinct umbrae represent the same
sequence of moments, in such case they are called {\it similar
umbrae}. More formally two umbrae $\alpha$ and $\gamma$ are {\it
similar} when $\alpha^n$ is umbrally equivalent to $\gamma^n,$ for
all $n=0,1,2,\ldots$ in symbols
$$\alpha \equiv \gamma \Leftrightarrow \alpha^n \simeq \gamma^n \quad
n=0,1,2,\ldots.$$ Given a sequence $1, a_1, a_2, \ldots$ in ${\cal R}$
there are infinitely many distinct, and thus similar umbrae,
representing the sequence.

Two umbrae $\alpha$ and $\gamma$ are said to be \textit{inverse}
to each other when $\alpha+\gamma\equiv\varepsilon.$ We denote the
inverse of the umbra $\alpha$ by $-1 \boldsymbol{.} \alpha.$ Note
that they are uncorrelated. Recall that, in dealing with a
saturated \footnote{Roughly speaking, a saturated umbral calculus
is defined when the alphabet $A$ is extended with a set including all auxiliary umbrae
like $-1 \punt \alpha.$ In \cite{SIAM}, a formal definition of saturated umbral
calculus is given.} umbral calculus, the inverse of an umbra is not unique,
but any two umbrae inverse to any given umbra are similar.

\subparagraph{The Bernoulli umbra.}

A definition of the Bernoulli umbra $\iota$ is given in \cite{SIAM}:
up to similarity, the Bernoulli umbra is the unique umbra such that
\begin{equation}
(\iota+1)^{n+1} \simeq \iota^{n+1}
\label{(main)}
\end{equation}
for all positive integers $n.$ Then, the Bernoulli umbra $\iota$ turns to be the
unique (up to similarity) umbra such that $E[\iota^n] = B_n,$ for $n=0,1,2,\ldots,$
where $\{B_n\}$ are the Bernoulli numbers. By using equivalence (\ref{(main)}), the
main properties of the Bernoulli numbers can be easily proved, as for example
$B_{2 n + 1} = 0$ for all nonnegative integers $n.$ Here, we just recall
that the {\it Bernoulli polynomials} $\{B_n(x)\}$ are such that
\begin{equation}
B_n(x) = \sum_{k=0}^n {n \choose k} B_k \, x^{n-k} \simeq (x +
\iota)^n. \label{(berpol)}
\end{equation}
The Bernoulli polynomials are characterized to have an average
value of $0$ over the interval $[0,1]$ for all nonnegative integers $n,$
that is
\begin{equation}
\int_0^1 B_n(x) \,{\rm d}x = 0.
\label{(condbis)}
\end{equation}
We sketch a simple \lq\lq{\it umbral \!}\rq\rq $\,$proof. Since by simple
computations we have $\int_0^1 B_n(x)  \,{\rm d}x
= E \left[ \int_0^1  (x + \iota)^n \,{\rm d}x \right],$ then
$\int_0^1 B_n(x) \,{\rm d}x \simeq \int_0^1 (x + \iota)^n \,{\rm d}x
\simeq [(\iota + 1)^{n+1} - \iota^{n+1}] / (n+1) \simeq 0,$ as
$(\iota+1)^{n+1} - \iota^{n+1} \simeq 0,$ due to equivalence (\ref{(main)}).
The approach here introduced allows us to manage integrals by using suitable
umbral polynomials.
\begin{thm}
If $-1 \punt \iota$ is the inverse of the Bernoulli umbra and
$\{B_n(x)\}$ are the Bernoulli polynomials, then
\begin{equation}
E[B_n(-1 \punt \iota)] = \int_0^1 B_n(x) \,{\rm d}x,
\label{(condI)}
\end{equation}
for all nonnegative integers $n.$
\end{thm}
\begin{proof}
Via equivalence (\ref{(berpol)}), we have $B_n(-1 \punt \iota)
\simeq (-1 \punt \iota + \iota)^n \simeq \varepsilon^n.$
\end{proof}
\begin{cor}
If $p(x) \in {\cal R}$ and $h \in {\mathbb R} \backslash \{0\}, c \in {\mathbb R}$ then
\begin{equation}
E[p(-1 \punt \iota)] = \int_0^1 p(u) \,{\rm d}u, \quad
E\left[p(-1 \punt (h \iota) + c) \right] = \frac{1}{h}
\int_c^{c+h} p(t) \, {\rm d}t. \label{(2)}
\end{equation}
\end{cor}
\begin{proof}
Without loss of generality, assume $p(x)$ a polynomial of degree
$n,$ for some nonnegative integer $n$, that is $p(x)= \sum_{k=0}^n
c_k x^k.$  By recalling that $E[(-1 \punt \iota)^k] = {1 / {(k+1)}}$ for
all nonnegative integers $k$ \cite{SIAM}, we have
$$E[p(-1 \punt \iota)] = \sum_{k=0}^n \frac{c_k}{k + 1} = \sum_{k=0}
^n c_k \int_0^1 x^k \,{\rm d}x = \int_0^1 p(x) \,{\rm d}x.$$
The latter of equations (\ref{(2)}) follows from the former. Indeed in
$\int_c^{c+h} p(t) \, {\rm d}t$ replace $t$ by $h \, u+c.$ The result
follows by recalling the equivalence $-1 \punt (h
\iota) \equiv h (-1 \punt \iota),$ proved in \cite{Dinardo} for
any umbra $\alpha \in A.$
\end{proof}
In the following, we give an umbral proof of the so-called
{\it multiplication theorem} \cite{Roman} for the umbral Bernoulli polynomials. We
will use this property in the next section.
\begin{thm}
If $m$ is a nonnegative integer, then for all nonnegative integers $n$
\begin{equation}
\left(x + {\iota \over m}\right)^n \simeq {1 \over m} \sum_{k=0}^{m-
1} \left( x + \frac{k}{m} + \iota \right)^n.
\label{(genmul)}
\end{equation}
\end{thm}
\begin{proof}
Since $-1 \punt \iota + \iota \equiv \varepsilon,$ we have
\begin{equation}
\left(x + {\iota \over m} \right)^n \!\! \simeq \!\! \left(  -1 \punt \iota
+ \frac{\iota}{m} + x + \iota \right)^n \!\! \simeq \!\! \sum_{j=0}^n {n \choose j} (x
+ \iota)^{n-j} \left( - 1 \punt \iota + \frac {\iota}{m}
\right)^j. \label{(2bis)}
\end{equation}
By using the former of equations (\ref{(2)}), we have
$$\left( - 1 \punt \iota + \frac{\iota}{m} \right)^{j}  \simeq
\int_0^1 \left( x + \frac{\iota}{m} \right)^{j} \, {\rm d}x \simeq
\sum_{k=0}^ {m-1} \int_{{k \over m}}^{\frac{k+1}{m}} \left( x +
\frac{\iota}{m} \right)^{j} \, {\rm d}x.$$ If we set $y = mx - k,$
then
$$\sum_{k=0}^{m-1} \int_{{k \over m}}^{\frac{k+1}{m}} \left( x + \frac
{\iota} {m} \right)^{j} \, {\rm d}x \simeq \frac{1}{m^{j+1}}
\sum_{k=0}^{m-1} \int_{0}^{1} (y + k + \iota)^{j} {\rm d}y \simeq
\frac{1}{m} \sum_{k=0}^{m-1} \left(\frac{k}{m} \right)^j.$$ The
result follows by substituting these last two equivalences in (\ref{(2bis)}).
\end{proof}
By linearity, if $p(x) \in {\cal R},$ then $p\left(x + {\iota
\over m}\right) \simeq {1 \over m} \sum_{k=0}^{m- 1} p\left(x +
\frac{k}{m} + \iota \right).$
\section{Corrections to grouped moments}
Usually, the $n$-th moment $\tilde{a}^{\prime}_n$ of the grouped distribution
is represented by
\begin{equation}
\tilde{a}^{\prime}_n = \sum_{i \in {\mathbb Z}}  \xi^n_i P\left( \xi_i -
\frac{h}{2} < X < \xi_i + \frac{h}{2} \right),
\label{(agrouped)}
\end{equation}
where $P(\cdot)$ denotes the parent distribution, $\xi_i$ are the
mid-point of classes partitioning the range and $h$ is the width
of each class. In practice, we know an estimate of
$\tilde{a}^{\prime}_n,$ since when $\tilde{a}^{\prime}_n$ is
computed, the probability
$$P\left( \xi_i - \frac{h}{2} < X < \xi_i
+ \frac{h}{2} \right)$$ is replaced by the frequency $N_i/N$ of the
corresponding class and only a finite number of classes is
considered (so the summation is over a finite number of terms). In
establishing approximate relations between the set of the raw
moments $a_n$ and the set of grouped moments
$\tilde{a}^{\prime}_n,$ a way to avoid any assumption other than
the existence of the involved moments is to employ another set of
constants $\tilde{a}_n.$ These constants $\tilde{a}_n$ are the
average of $\tilde{a}^{\prime}_n,$ when the set $\{\xi_i\}$ is
replaced by a suitable set of random points obtained by assuming
random the rounding lattice. By the umbral method, we will prove
that the expression of the raw moments $a_n$ in terms of the
constants $\tilde{a}_n$ gives the corrections to grouped moments.
The discussion on the nature of approximation in replacing
$\tilde{a}_n$ by $\tilde{a}^{\prime}_n$ goes beyond the aim of
this paper, and it has been already tackled in the literature, see
for example \cite{Abernethy, Heitjan} and \cite{Wilson}.
\subparagraph{Continuous parent distribution: Sheppard's corrections.}
Let $f$ be a continuous probability density function of a $\rv \, X$
over $(-\infty, \infty).$ As usual,
\begin{equation}
a_n = \int_{-\infty}^{\infty} t^n f(t) \,{\rm d}t
\label{(momc)}
\end{equation}
denotes the $n$-th moment of $X$ about the origin. In the following, we assume that
all absolute moments exist. The moments
calculated from the grouped frequencies are given by
\begin{equation}
\tilde{a}_n = {1 \over h} \int_{-\infty}^{\infty} t^n \int_{- {1 \over
2} h}^{{1 \over 2} h} f(t + x) \,{\rm d}x \,{\rm d}t.
\label{(infinite)}
\end{equation}
Indeed in (\ref{(agrouped)}), suppose to replace $\xi_i$ by $\xi_i + U,$
with $U$ an uniform $\rv$ over $(-{1 \over 2} h, {1 \over 2} h).$  Equation
(\ref{(infinite)}) follows by setting $\tilde{a}_n = E[\tilde{a}^{\prime}_n(U)].$
\begin{thm}[Sheppard's correction] \label{2} If the sequence $\{\tilde{a}_n\}$ in
(\ref{(infinite)}) is umbrally represented by the umbra $\tilde{\alpha}$ and the sequence $\{a_n\}$
in (\ref{(momc)}) is umbrally represented by the umbra $\alpha,$ then
\begin{equation}
\tilde{\alpha} \equiv  \alpha + h \left(-1 \punt \iota  - \frac{1}{2} \right).
\label{(shc)}
\end{equation}
\end{thm}
\begin{proof} We have
$$E[(\alpha+x)^n] = \sum_{k=0}^n {n \choose k} E[\alpha^k] \, x^{n-k} = \int_{-\infty}^{\infty}
\sum_{k=0}^n {n \choose k} t^k \, x^{n-k} \, f(t) \,{\rm d}t,$$
so that $(\alpha+x)^n \simeq \int_{-\infty}^{\infty} (t+x)^n f(t) \,{\rm d}t,$
for all nonnegative integers $n.$ Since
$$\tilde{a}_n  =  {1 \over h} \int_{-\infty}^ {\infty}
\int_{- {1 \over 2} h}^{{1 \over 2} h} (t+ x)^n f(t) \,{\rm d}t \,{\rm d}x = {1
\over h} \int_{- {1 \over 2} h}^ {{1 \over 2} h}
E[(\alpha + x)^n] \,{\rm d}x,$$
and due to the linearity of $E$
$${1 \over h} \int_{- {1 \over 2} h}^ {{1 \over 2} h}
E[(\alpha + x)^n] \,{\rm d}x = E \left[ {1 \over h} \int_{- {1 \over 2} h}^ {{1 \over 2} h}
(\alpha + x)^n \,{\rm d}x \right],$$
equation (\ref{(shc)}) follows by observing
$$  {1 \over h} \int_{- {1 \over 2} h}^ {{1 \over 2} h}
(\alpha + x)^n \,{\rm d}x  \simeq  \left(- 1 \punt (h \iota) + \alpha - \frac{1} {2} h \right)^n,$$
where the last equivalence follows from the latter of (\ref{(2)}) with $c=-h/2.$
\end{proof}
As a first corollary, we get equation (\ref{(Sh1)}). Note
that from equivalence (\ref{(shc)}), we have
\begin{equation}
\alpha \equiv \tilde{\alpha} + h \left( \iota + \frac{1}{2} \right).
\label{(ii)}
\end{equation}
Then, by using the binomial expansion and by applying the linear
functional $E,$ for all nonnegative integers $n$ we have
$$a_n = \sum_{j=0}^n {n \choose j} E\left[ \left(\iota +
{1 \over 2} \right)^{j} \right] \, h^j \, \tilde{a}_{n-j}.$$
In order to recover equation (\ref{(Sh1)}), all we need is to prove that
$E\left[ \left(\iota + {1 \over 2} \right)^{j} \right] = B_j (2^{1-
j} - 1)$ for all nonnegative $j.$ This is done in the following proposition.
\begin{prop} \label{1} If $\iota$ is the Bernoulli umbra, then for all nonnegative $j$
\begin{equation}
\left(\iota + {1 \over 2} \right)^{j} \simeq (2^{1-j} - 1) \, \iota^j.
\label{(propb)}
\end{equation}
\end{prop}
\begin{proof}
Note that for all nonnegative integers $j$ we have
\begin{eqnarray}
\left(\iota + {1 \over 2} \right)^{j} & \simeq & \left(- 1 \punt {\iota^{\prime} \over 2} +
{\iota^{\prime} \over 2} + \iota + {1 \over 2} \right)^{j}
\simeq \left[\iota + {1 \over 2} \left( - 1 \punt \iota^{\prime} + 1
\right) +
{\iota^{\prime} \over 2} \right]^{j} \nonumber \\
& \simeq & \sum_{k=0}^j {j \choose k} \left( \iota + {\iota^{\prime}
\over 2}
\right)^{j-k} {1 \over 2^k}  \left( - 1 \punt \iota^{\prime} + 1 \right)
^k,
\label{(proof1)}
\end{eqnarray}
where $\iota$ and $\iota^{\prime}$ denote uncorrelated Bernoulli umbrae.
For all nonnegative integers $k,$ from equation (\ref{(2)}) we have
$$ \left( - 1 \punt \iota^{\prime}  + 1 \right)^k \simeq \int_{0}^1 (x+1)
^k {\rm d} x = {2^{k+1} \over {k+1}} - {1 \over {k+1}} \simeq
2^{k+1} (-1 \punt \iota)^k - (-1 \punt \iota^{\prime})^k.$$
Substituting this last equivalence in (\ref{(proof1)}), we have
\begin{eqnarray*}
\left(\iota + {1 \over 2} \right)^{j} & \simeq &
 \sum_{k=0}^j {j \choose k} \left( \iota + {\iota^{\prime} \over 2}
\right)^{j-k} \left[2 \, (-1 \punt \iota)^k -  \left(-1 \punt {\iota^{\prime}
\over 2}\right)^k \right] \\
& \simeq & 2  \sum_{k=0}^j {j \choose k} (-1 \punt \iota)^k \left(
\iota + {\iota^{\prime} \over 2} \right)^{j-k}  -
\sum_{k=0}^j {j \choose k} \left(-1 \punt {\iota^{\prime} \over 2}
\right)^k \left({\iota^{\prime} \over 2} + \iota \right)^{j-k} \\
& \simeq & 2 \left( -1 \punt \iota + \iota + {\iota^{\prime} \over 2} \right)^{j} -
\left( -1 \punt {\iota^{\prime} \over 2} + {\iota^{\prime} \over 2} + \iota \right)^{j} \simeq 2 \left(  {\iota^{\prime} \over 2} \right)
^{j} - \iota^{j},
\end{eqnarray*}
by which equivalence (\ref{(propb)}) follows.
\end{proof}
As a second corollary of Theorem \ref{2}, we recover equations giving
the $n$-th moment of grouped data in terms of raw moments.
Indeed, in equivalence (\ref{(shc)}), by using the binomial expansion and by applying the linear
functional $E,$ we have
\begin{equation}
\tilde{a}_n = \sum_{j=0}^n {n \choose j} E\left[ \left(-1 \punt \iota
-{1 \over 2} \right)^{j} \right] \, h^j \, a_{n-j},
\label{(Sh)}
\end{equation}
for all nonnegative integers $n.$ As before, we need to evaluate
the moments of $-1 \punt \iota -{1 \over 2}.$ This can be done
by using only equation (\ref{(2)}). Indeed, for all nonnegative
integers $j$ we have
\begin{equation}
E\left[\left(-1 \punt \iota -{1 \over 2} \right)^{j} \right] =
\int_0^1 \left (x - {1 \over 2} \right)^{j} {\rm d} x
= \left\{ \begin{array}{ll}
0, & \hbox{if $j$ is odd}, \\
\frac{1}{j+1} \left({1 \over 2}\right)^{j}, & \hbox{if $j$ is even}.
\end{array} \right. \label{(imp)}
\end{equation}
So from equation (\ref{(Sh)}) and by using (\ref{(imp)}), we have
\begin{equation}
\tilde{a}_n  = \sum_{j=0}^{\lceil n/2 \rceil} {n \choose 2j} \left(
\frac{h}{2} \right)^{2j} \frac{a_{n - 2j}}{2 j + 1}.
\label{(Sh3)}
\end{equation}
\begin{rem}
{\rm Theorem \ref{2} still holds, when the moments are referred to a parent distribution
over $(a, b).$ In (\ref{(momc)}) and (\ref{(infinite)}), instead of using the domain of integration
$(-\infty, \infty),$ we integrate over $(a,b)$ and we do the same
in the proof of Theorem \ref{2}. In (\ref{(agrouped)}), we refer the summation
to $i=1, \ldots, p,$ with $p$ the number of classes partitioning $(a,b)$ with
width $h$. Here the $\rv \, \tilde{a}^{\prime}_n(U)$ is obtained by replacing $\xi_i$ with
$a + (i - {1 \over 2}) h + U.$}
\end{rem}
\subparagraph{Discrete parent distribution.}
Assume that $m$ equidistant consecutive values of a discrete $\rv$ are grouped
into a frequency class of width $h.$ The $m$ smaller intervals of width
$h/m$ go to make up the class width $h$ in such a way that the $m$
values of the variable represent the mid-points of the sub-intervals.
Without loss of generality, we assume that
$$P \left( X = {i h \over m} \right) = {h \over m} f\left( {i h \over
m} \right), \qquad i \in {\mathbb Z},$$
with
$$ f\left( {i h \over m} \right) \geq 0 \quad \forall i \in {\mathbb
Z} \quad \hbox{\rm and} \quad \sum_{i \in {\mathbb Z}}{h \over m}
f\left( {i h \over m} \right) = 1.$$
In (\ref{(agrouped)}), replace $\xi_i$ by
$$\left(m \, i + U - {{m-1} \over 2}\right) {h \over m},$$
with $U$ a $\rv$ which has any of $m$ possible values
$0,1,\ldots, m-1$ equally probable.
The $m$ values of $\tilde{a}^{\prime}_n(U)$ corresponding to the $m$ distinct
methods of grouping a discrete distribution are
$$\tilde{a}^{\prime}_n(U) = \sum_{i \in {\mathbb Z}} \left[ \left( m \, i + U -
{{m-1} \over {2}} \,
\right) {h \over m} \right]^n
\sum_{j=0}^{m-1} {h \over m} f\left[ \left( m \, i + U - j \right) {h
\over m} \right].$$
By doing some calculations, we recover the expression of moments calculated from the grouped frequencies
\begin{equation}
\tilde{a}_n=E[\tilde{a}^{\prime}_n(U)] = {1 \over m} \sum_{j=0}^{m-1} {h \over m}
\sum_{s \in {\mathbb Z}} \left( {s h \over m}
- {{m-1-2j} \over {2m}} \, h \right)^n  f\left( {s h \over m} \right).
\label{(momdisgr)}
\end{equation}
In the following, we denote by $a_n$ the $n$-th moment of the discrete
parent distribution, that is
\begin{equation}
a_n= {h \over m} \sum_{s \in {\mathbb Z}} \left(s h \over m\right)^n
f\left( {s\over m} \right).
\label{(momdis)}
\end{equation}
\begin{thm}  [Corrections] \label{2bis} If the sequence $\{\tilde{a}_n\}$
in (\ref{(momdisgr)}) is umbrally represented by the umbra $\tilde{\alpha}$ and
the sequence $\{a_n\}$ in (\ref{(momdis)}) is umbrally represented by the
umbra $\alpha,$ then
\begin{equation}
\tilde{\alpha} \equiv  \alpha  + h \left(- 1 \punt \iota - {1 \over 2} \right) + {h \over m} \left( \iota +
{1 \over 2} \right).
\label{(shcdis)}
\end{equation}
\end{thm}
\begin{proof}
By linearity, we have
$$E\left[\left(\alpha - {{m-1-2j} \over {2m}} \, h\right)^n\right] =
{h \over
m} \sum_{s \in {\mathbb Z}} \left( {s h \over m}
- {{m-1-2j} \over {2m}} \, h \right)^n  f\left( {s h \over m}
\right).$$
The result follows by observing that
\begin{eqnarray}
\tilde{\alpha}_n  & \simeq &   {1 \over m} \sum_{j=0}^{m-1} \left(
\alpha - {{m-1-2j} \over {2m}} \, h \right)^n \nonumber \\
& \simeq & {h^n \over m} \sum_{j=0}^{m-1} \left( \left\{  - 1 \punt \iota + {\alpha \over
h} - {{m-1} \over {2m}} \right\} + {j \over m}  + \iota \right)^n
\nonumber \\
& \simeq & h^n \left(- 1 \punt \iota + {\alpha \over h}  - {{m-1}
\over {2m}} + {\iota \over m}  \right)^n,
\label{(eq1)}
\end{eqnarray}
where equivalence (\ref{(eq1)}) follows from equivalence (\ref{(genmul)}), by replacing
$x$ with $\left\{ - 1 \punt \iota + {\alpha \over h}  - {{m-1} \over {2m}}
\right\}.$ Suitably rearranging the terms, we obtain equivalence (\ref{(shcdis)}).
\end{proof}
It is interesting to compare equivalence (\ref{(shc)}) with
(\ref{(shcdis)}). In this last equivalence, we find just one
addend more, which is an umbra whose moments are given in
Proposition \ref{1}. Obviously, if $m \rightarrow \infty$ from
(\ref{(shcdis)}) we recover (\ref{(shc)}), as stated by Craig in
\cite{Craig}. Moreover, since from (\ref{(shcdis)}) we find
\begin{equation}
\alpha \equiv \tilde{\alpha} + h \left( \iota + {1 \over 2} \right) + {h \over m} \left( -1 \punt \iota
- {1 \over 2} \right),
\label{(iii)}
\end{equation}
equivalence (\ref{(iii)}) differs from equivalence (\ref{(ii)}) for an umbra whose moments are given in
(\ref{(imp)}). These observations turn to be useful when we seek expression of raw moments $\{a_n\}$
in terms of grouped moments $\{\tilde{a}_n\}$. To the best of our knowledge, the last version available
is given by Craig in \cite{Craig}, but its structure is quite complex and involves
integer partitions. Here we give a different expression. By applying the binomial
expansion and the linear functional $E$ to equivalence (\ref{(iii)}), for all nonnegative
integers we have
\begin{equation}
a_n = \sum_{k=0}^n {n \choose k} E \left[ \left\{\tilde{\alpha} + h \left( \iota + {1 \over 2} \right) \right\}^{n-k} \right]
E \left[ \left( -1 \punt \iota - {1 \over 2} \right)^{k} \right] \frac{h^k}{m^k}.
\label{(grdis)}
\end{equation}
Moments of $\tilde{\alpha} + h \left( \iota + {1 \over 2} \right)$ are given by Sheppard's corrections (\ref{(Sh1)})
in the continuous case, while the moments of  $-1 \punt \iota - {1 \over 2}$ are given in (\ref{(imp)}). So
all we need is to replace these expressions in equation (\ref{(grdis)}):
$$a_n = \sum_{k=0}^{\lceil {n \over 2} \rceil} {n \choose {2k}} \left( \frac{h}{2m} \right)^{2k} \frac{1}{2k+1}
\sum_{j=0}^{n- 2k} {{n-2k} \choose j} \left( 2^ {1-j} -1 \right) B_j \, h^j \, \tilde{a}_{n - 2k - j}.$$
By a {\tt Maple} procedure, it is straightforward to verify that the first eight moments $\{a_1, \ldots, a_8\},$
computed by means of this formula, are the same as given by Craig in \cite{Craig}.
\section{Corrections to multivariate grouped data}
In \cite{Bernoulli}, it has been shown that the notion of multiset is the key for
dealing with multivariate moments in umbral syntax. Here we recall briefly
the notations.

Let $M$ be a multiset of umbral monomials. A {\it multiset} $M$ is a pair
$(\bar{M}, f),$ where $\bar{M}$ is a set, called the {\it support} of the
multiset, and $f$ is a function from $\bar{M}$ to nonnegative integers.
For each $\mu \in \bar{M},$ $f(\mu)$ is the {\it multiplicity} of $\mu.$
We denote a multiset $(\bar{M}, f)$ simply by $M.$
When the support of $M$ is a finite set, say $\bar{M}=\{\mu_1, \mu_2,
\ldots, \mu_j\},$ we write
$$M = \{\mu_1^{(f(\mu_1))},\mu_2^{(f(\mu_2))}, \ldots, \mu_j^{(f
(\mu_j))}\} \quad \hbox{or} \quad M = \{\underbrace{\mu_1, \ldots,
\mu_1}_{f (\mu_1)}, \ldots, \underbrace{\mu_j, \ldots,
\mu_j}_{f(\mu_j)}\}.$$
Set
\begin{equation}
\mu_{M} =  \prod_{\mu \in \bar{M}} \mu^{f(\mu)}. \label{(momb4)}
\end{equation}
For example, if $M=\{\mu_1^{(2)}, \mu_2^{(1)}, \mu_3^{(4)} \},$ we denote by $\mu_M$ the product
$\mu_1^2 \mu_2 \mu_3^4.$
\par
 A {\it multivariate moment} is the element of ${\cal R}$ corresponding to the
umbral monomial $\mu_M$ via the evaluation $E,$ i.e.
\begin{equation}
E[\mu_M] = m_{t_1 \ldots \, t_j},
\label{(not)}
\end{equation}
where $t_i=f(\mu_i)$ for $i=1,2,\ldots,j.$
For example if $M=\{\mu_1^{(2)}, \mu_2^{(1)}, \mu_3^{(4)} \},$ we
have $E[\mu_{M}]=m_{2 1 4}.$ When the umbral monomials $\mu_i$ are
uncorrelated, $m_{t_1 \ldots \, t_j}$ becomes the product of
the moments of $\mu_i.$ More details on the meaning and the use of the symbol
$\mu_M$ are given in \cite{Bernoulli}.
\par
Suppose ${\bf X}=(X_1, X_2, \ldots, X_j)$ a multivariate $\rv$ with the joint density
function $f_{\bf X}({\bf x})$ over ${\mathbb R}^j.$ Note that by using the same arguments,
we can deal with any range of bounded rectangle type. As usual
\begin{equation}
m_{t_1 \ldots \, t_j} = \int_{{\mathbb R}^j} x_1^{t_1} \cdots x_j^{t_j} f_{\bf X}({\bf x}) \, {\rm d}{\bf x}
\label{(mulmom)}
\end{equation}
denotes the multivariate moment of ${\bf X}$ of order $(t_1, \ldots, t_j).$ The moments calculated
from the grouped frequencies are given by
\begin{equation}
\tilde{m}_{t_1 \ldots \, t_j} = \frac{1}{h_1 \cdots h_j} \int_{{R}_j} \int_{{\mathbb R}^j}
(x_1 + z_1)^{t_1} \cdots (x_j + z_j)^{t_j} f_{\bf X}({\bf x}) \, {\rm d}{\bf x} \, {\rm d}{\bf z},
\label{(mulgr)}
\end{equation}
where
$$R_j = \left\{ {\bf z}= (z_1, \ldots, z_j) \in {\mathbb R}^j : z_k \in \left(-{1 \over 2} h_k, {1 \over 2} h_k \right) \, \forall k \in
\{1,\ldots,j\} \right\}$$
and $\{h_k\} \in {\mathbb R} \backslash \{0\}$ are the width window for any component. A proof of $(\ref{(mulgr)})$ can be done similarly
to the one sketched for the univariate case $(\ref{(agrouped)}).$
\begin{thm}[Multivariate Sheppard's correction] \label{4} If the sequence $\{\tilde{m}_{t_1 \ldots \, t_j}\}$ in
(\ref{(mulgr)}) is umbrally represented by the umbral monomial $\tilde{\mu}_M,$ with $M$ a multiset of finite
support $\{\mu_1, \ldots,\mu_j\},$ and the sequence $\{m_{t_1 \ldots \, t_j}\}$ in (\ref{(mulmom)}) is umbrally represented
by the umbral monomial $\mu_M,$ then
\begin{equation}
\tilde{\mu}_M \equiv  \left[\mu + h \left(-1 \punt \iota  - \frac{1}{2} \right) \right]_M
\label{(shcmul)}
\end{equation}
where we set
$$
\left[\mu + h \left(-1 \punt \iota  - \frac{1}{2} \right) \right]_M = \prod_{k=1}^j \left[ \mu_k + h_k \left( -1 \punt \iota_k
- \frac{1}{2} \right) \right]^{t_k}
$$
with $\{\iota_k\}$ uncorrelated Bernoulli umbrae.
\end{thm}
\begin{proof} Due to linearity, we have
$$E[(\mu_1 + z_1)^{t_1} \cdots (\mu_j + z_j)^{t_j}] = \int_{{\mathbb R}^j} (x_1 + z_1)^{t_1} \cdots
(x_j + z_j)^{t_j} f_{\bf X}({\bf x}) \, {\rm d}{\bf x}.$$
Moreover, from (\ref{(mulgr)}) we have
\begin{eqnarray*}
\tilde{m}_{t_1 \ldots \, t_j} &  \simeq &  \frac{1}{h_1 \cdots h_j} \int_{{R}_j}
(\mu_1 + z_1)^{t_1} \cdots (\mu_j + z_j)^{t_j} {\rm d}{\bf z}, \\
& \simeq & \prod_{k=1}^j {1 \over h_k} \int_{-{1 \over 2} h_k}^{{1 \over 2}h_k} (\mu_k + z_k)^{t_k} {\rm d} z_k
\simeq \prod_{k=1}^j  \left( \mu_k + h_k \left( -1 \punt \iota_k
- \frac{1}{2} \right) \right)^{t_k}
\end{eqnarray*}
by which the result follows.
\end{proof}
In the support of the multiset $M,$ if we choose
$$\mu_k = \tilde{\mu}_k + h_k \left( \iota_k + \frac{1}{2} \right), \,\, k=1, \ldots, j$$
then equivalence $(\ref{(shcmul)})$ becomes an identity and so by notation (\ref{(not)})
\begin{equation}
\mu_M \equiv  \left[\tilde{\mu} + h \left( \iota + \frac{1}{2} \right) \right]_M
\label{(shcmul2)}
\end{equation}
where again
$$\left[\tilde{\mu} + h \left( \iota + \frac{1}{2} \right) \right]_M = \prod_{k=1}^j \left[ \tilde{\mu}_k + h_k \left( \iota_k
+ \frac{1}{2} \right) \right]^{t_k}.$$
Equivalence (\ref{(shcmul2)}) can be easily implemented in any symbolic package. Indeed,
in order to recover expressions of raw multivariate moments in terms of grouped moments,
we need to multiply summations like
$$\sum_{s_k=1}^{n} {n \choose s_k} \tilde{\mu}_k^{s_k} h_k^{n - s_k} \left( 2^{1-n+s_k} - 1\right) B_{n-s_k}$$
corresponding to the $n$-th power of $\tilde{\mu}_k + h_k \left( \iota_k
+ \frac{1}{2} \right),$ and then replace occurrences of products like $\tilde{\mu}_1^{s_1}
\tilde{\mu}_2^{s_2} \cdots \tilde{\mu}_j^{s_j}$ with $\tilde{m}_{s_1 \ldots s_j}.$
\par
For the sake of brevity, we skip the details of the proof of corrections when
the multivariate parent distribution is discrete. This can be done taking the
same way used for the univariate parent distribution, as we have done for the
continuous case. Here, the corrections to moments due to the grouping
can be formulated in umbral terms as:
\begin{equation}
\mu_M \equiv \left[ \tilde{\mu} + h \left( \iota + {1 \over 2} \right) + {h \over m} \left( -1 \punt \iota
- {1 \over 2} \right) \right]_M,
\label{(iv)}
\end{equation}
where, by the symbol in the right hand side of the previous equivalence, we denote the following
product
$$ \prod_{k = 1}^j  \left[ \tilde{\mu}_k + h_k \left( \iota_k + {1 \over 2} \right) + {h_k \over m_k} \left( -1 \punt \iota_k
- {1 \over 2} \right) \right],$$
where $m_k$ are the number of consecutive values grouped in a frequency class of width $h_k.$
The multivariate version of similarity (\ref{(shcdis)}) is
\begin{equation}
 \tilde{\mu}_M \equiv \left[ \mu + {h \over m} \left( \iota + {1 \over 2} \right) +  h \left( -1 \punt \iota
- {1 \over 2} \right) \right]_M.
\label{(v)}
\end{equation}
\subparagraph{Concluding remarks.}
The main goal of this paper is to show how the corrections of moments resulting
from grouping into classes may be summarized in few closed-form formulae.
Moreover, the multivariate formulae can be constructed from the univariate ones,
by a suitable indexing of umbral monomials with a multiset.

Once more, this paper shows how the classical umbral calculus
should be taken into account for managing sequence of numbers
related to $\rv$'s, since  many calculations are reduced. For
example, the reader interested in recovering corrections for
cumulants and factorial moments, by using the classical umbral
calculus, can refer to \cite{Dinardoeurop}. The noteworthy
simplification in the expression of corrections, when referred to
cumulants instead of moments, was first pointed out by Langdon and
Ore \cite{Langdon} for a continuous parent distribution over $(-
\infty, \infty)$. By using the umbral syntax introduced for the
$\alpha$-cumulant umbra and the $\alpha$-factorial umbra, it is
possible to recover these corrections by one line proof both for
continuous and for discrete parent distributions.

The umbral techniques applied to Sheppard's corrections open the
way to deal with new problems that would be interesting to
explore. Indeed, the umbral version of Sheppard's corrections,
here introduced, refers to the averaging interpretation of these
corrections, as proposed by Craig \cite{Craig}. Different
interpretations give rise to different forms of Sheppard's
corrections, see for example \cite{Schneeweiss}.  It would be
interesting to see if the umbral method simplifies the calculation
apparatus and add some new formulae also for these different
interpretations.

\end{document}